\documentclass{amsart} 

\usepackage{hk-roots}

\usepackage[pdftex]{hyperref}
\hypersetup{colorlinks,citecolor=blue,linktocpage,hyperindex=true,backref=true}
\date{June 19, 2022}
\title{Root systems and hyperk\"ahler varieties}
\author{Valery Alexeev}
\newcommand\Sym{\operatorname{Sym}}

\begin{document}

\begin{abstract}
  We point out a connection between root systems and some of the known
  hyperk\"ahler varieties.
\end{abstract}

\maketitle 

Let $\Lambda$ be a root system with Weyl group $W$ and $G$ be a
group variety. Varieties of the form $(\Lambda\otimes G)/W$ and
$\Hom(\Lambda, G)/W$ frequently appear as strata in various moduli
spaces. For example they appear implicitly in
\cite{losev2000new-moduli-spaces} for $\Lambda=A_n$ and $G=\bC^*$ and
explicitly in \cite{alexeev17ade-surfaces, 
  alexeev2020compactifications-moduli} for $\Lambda=A_n, D_n, E_n$ and
$G=\bC^*$ or an elliptic curve. What happens if $G=A$ is an abelian
surface?

\begin{example}
  Let $\Lambda=A_n$ and $X=(A_n\otimes A)/W(A_n) = \Hom(A_n^*,
  A)/W(A_n)$. Since $A_n^* = \bZ^{n+1}/\diag\bZ$, 
  \begin{displaymath}
    \{(g_1,\dotsc, g_{n+1}) \in G^{n+1} \mid \sum g_i=0 \}  
  \end{displaymath}
  One has $W(A_n) = S_{n+1}$, thus
  $X = \ker \big( \Sym^{n+1}(A) \xrightarrow{\sum} A\big)$, which is
  birational to a generalized Kummer variety $K_n(A)$.
\end{example}

\begin{example}
  Let $\Lambda=B_n$ and $X=(B_n\otimes A)/W(B_n)$. One has $B_n=\bZ^n$
  and $W(B_n) = \bZ_2^n \rtimes S_n$ acts on it in a natural way. Then
  the quotient $(B_n\otimes A)/\bZ_2^n = K^n$, where $K=A /\bZ_2$ is
  the Kummer surface of $A$, and $X = \Sym^n(K)$ is birationally
  isomorphic to a particular HK variety of $\Hilb^n({\rm K3})$ type.
\end{example}

\begin{question}
  Are there root lattice constructions for the O'Grady OG6 and OG10
  varieties and, optimistically, perhaps for other smooth hyperkahler
  varieties? 
\end{question}

Let us generalize the above construction a little.

\begin{lemma}
  Let $W$ be a finite (not necessarily reflection) group and
  let $L\simeq\bZ^n$ be an irreducible integral 
  $W$-representation. Then the subspace of $W$-invariants in
  $\wedge^2(L\oplus L)$ is isomorphic to $\bZ$.
\end{lemma}
\begin{proof}
  Let $V=L\otimes\bQ$. One has
  \begin{displaymath}
    \wedge^2(V\oplus V)  = (\wedge^2 V)^{\oplus 3}
    \oplus \Sym^2(V),
  \end{displaymath}
  We claim that and $(\wedge^2 V)^W=0$ and
  $\left( \Sym^2 V\right)^W=\bQ$, from which the statement follows.
  Indeed, an $W$-invariant bilinear form is the same as an isomorphism
  of $W$-representations $V\to V^*$.  Since there exists a
  $W$-invariant symmetric bilinear form on $V$ (just take any positive
  definite symmetric form and average it), $V^*\simeq V$. By Schur's
  lemma the bilinear form is unique up to a scalar. Since it is
  symmetric, it lies in $\Sym^2V$ and there are none in $\wedge^2V$.
\end{proof}

\begin{theorem}
  Let $W$ and $L$ be as in the lemma, and let $A$ be an abelian surface.
  Define an abelian variety
  $Y = L\otimes A = \Hom(L^*, A)$ and  let $X = Y/W$. Then the map $Y\to X$
  is unramified outside of a subset of codimension $\ge2$ and the
  smooth locus $U$ of $X$ carries a nondegenerate holomorphic
  symplectic form unique up to a constant (so that $X$ is a singular
  IHS in an appropriate definition).
\end{theorem}
\begin{proof}
  Let $y\in L\otimes A$ be a point with a nontrivial stabilizer
  $G\subset W$. The action of $G$ on the tangent space $T_{y, Y}$ is
  identified with a restriction of the action of $W$ on
  $T_{0,Y} = L\otimes T_{0,A}$. The latter action is free outside a
  subset of codimension $2$ because it splits into two actions on
  $L\otimes\bC$, both with nontrivial kernel.  This shows that
  $W$-action is free outside a subset of codimension $\ge2$.

  Any nondegenerate holomorphic symplectic form on $U$ extends to a
  $W$-invariant nondegenerate holomorphic symplectic form on $Y$ and
  vice versa. Now,
  \begin{displaymath}
    H^0(Y, \Omega^1) = L\otimes \bC^2 
    \text{ and }
    H^0(Y,\Omega^2) = \wedge^2(L\oplus L) \otimes \bC.
  \end{displaymath}
  By the above lemma, $\left( \wedge^2(L\oplus L) \right)^W = \bZ$, so
$H^0(U, \Omega^2) = \bC$.
\end{proof}

\begin{remark}
  The easiest case where the theorem applies is when $L$ is a finite
  index $W$-invariant sublattice of a root lattice $\Lambda$. In
  that case, the singularity of $X$ at the origin is
  $(\Lambda\otimes \bC^2)/W$.  By \cite{kuznetsov2007quiver-varieties,
    ginzburg2004poisson-deformations} this singularity has a
  symplectic resolution if and only if $\Lambda$ is the root lattice
  of types $A$, $B$ or $C$ (and not $D,E,F,G,H$).
  
  The lattices invariant under $W(A_n)$ and $W(B_n)=W(C_n)$ are easy
  to list: up to rescaling for $A_n$ they are $A_n\subset
  L\subset A_n^*$, and for $B_n,C_n$ they are $D_n$, $B_n$ and
  $D_n^*$. 
  It appears that they provide no other examples other than
  the two above which admit a symplectic resolution. 
\end{remark}

\begin{acknowledgments}
  I thank Nikon Kurnosov and Claire Voisin for useful conversations.
\end{acknowledgments}

\bibliographystyle{amsalpha}
\def\cprime{$'$}
\providecommand{\bysame}{\leavevmode\hbox to3em{\hrulefill}\thinspace}
\providecommand{\MR}{\relax\ifhmode\unskip\space\fi MR }
\providecommand{\MRhref}[2]{%
  \href{http://www.ams.org/mathscinet-getitem?mr=#1}{#2}
}
\providecommand{\href}[2]{#2}

\end{document}